\documentclass[12pt]{amsart}  

\usepackage[latin1]{inputenc}
\usepackage{amsmath} 
\usepackage{amsfonts}
\usepackage{amssymb}
\usepackage{stmaryrd}
\usepackage{latexsym} 
\usepackage{graphicx}
\usepackage{subfigure}
\usepackage{color}
\usepackage{hyperref}
\usepackage{verbatim}
\usepackage[all]{xy}
\usepackage{graphics}
\usepackage{pdfsync}
\usepackage{xcolor}
\usepackage{tikz}
\usepackage{bm}
\usetikzlibrary{arrows.meta}
\usetikzlibrary{shapes.geometric}

\theoremstyle{definition}
\newtheorem{theorem}{Theorem}[section]

\newtheorem{lemma}[theorem]{Lemma}
\newtheorem{corollary}[theorem]{Corollary}

\newtheorem{conjecture}[theorem]{Conjecture}
\newtheorem{definition}[theorem]{Definition}
\newtheorem{example}[theorem]{Example}

\theoremstyle{remark}
\newtheorem{remark}[theorem]{Remark}

\numberwithin{equation}{section}
\begin{document}

\title{Linear relations on star coefficients of the chromatic symmetric function}

\author[Orellana]{Rosa Orellana}
\address[R. Orellana]{Mathematics Department, Dartmouth College,
Hanover, NH 03755, U.S.A.}
\email{Rosa.C.Orellana@dartmouth.edu}
\urladdr{\href{https://math.dartmouth.edu/~orellana}{https://math.dartmouth.edu/~orellana}}
\author[Tom]{Foster Tom}
\address[F. Tom]{Mathematics Department, Dartmouth College,
Hanover, NH 03755, U.S.A.}
\email{ftom@berkeley.edu}
\urladdr{\href{https://math.dartmouth.edu/~ftom/}{https://math.dartmouth.edu/~ftom}}
\subjclass[2020]{Primary 05C05; Secondary 05E05}
\keywords{}


\begin{abstract}
We prove that the coefficient of $\mathfrak{st}_{21^{n-2}}$ in the star basis of the chromatic symmetric function $X_G$ determines whether a connected graph $G$ is $2$-connected. We also prove new linear relations on other star coefficients of chromatic symmetric functions. This allows us to find new bases for certain spans of chromatic symmetric functions. Finally, we relate the coefficient of the star $\mathfrak{st}_n$ to acyclic orientations.
\end{abstract}

\maketitle

\section{Introduction}\label{section:introduction}

In 1995, Stanley \cite{chromsym} introduced the \emph{chromatic symmetric function} $X_G(\bm x)$ defined for any simple graph $G=(V,E)$ and set of commuting variables ${\bm x}=(x_1, x_2, \ldots$). The chromatic symmetric function is a generalization of the \emph{chromatic polynomial} $\chi_G(t)$, which can be recovered by setting $\chi_G(t)=X_G(\underbrace{1,\ldots,1}_{t\text{ ones}},0,\ldots)$. Although the chromatic polynomial $\chi_G(t)$ can tell us whether $G$ is connected, $2$-connected \cite{cutchrom}, or whether $G$ is a tree, it is completely unable to distinguish trees because $\chi_T(t)=t(t-1)^{n-1}$ for all $n$-vertex trees $T$. By contrast, Stanley observed that trees can have different chromatic symmetric functions, which has led to the following conjecture: 

\begin{conjecture}[Tree Isomorphism Conjecture] \cite{chromsym} If $T_1$ and $T_2$ are non-isomorphic trees, then $X_{T_1}(\bm x)\neq X_{T_2}(\bm x)$.
\end{conjecture}

This conjecture has been checked for all trees up to 29 vertices \cite{algochromtrees} and for the classes of caterpillars \cite{catchromiso}, spiders \cite{distinguishingtrees}, and trees with diameter at most 5 \cite{chromstar}. The chromatic symmetric function of a tree determines its subtree polynomial \cite{distinguishingtrees} and its generalized degree polynomial \cite{chromtreepolyinv, vertweightgen}, which in turn determines information such as its double-degree sequence and leaf adjacency sequence \cite{gendegpol}. For general graphs, the chromatic symmetric function encodes information about acyclic orientations \cite{chromsym} and connected partitions \cite{conparts}. \\

In this paper, we use the deletion-near-contraction relation of Aliste-Prieto, de Mier, Orellana, and Zamora \cite{markedgraphs} to study the expansion $X_G=\sum_\lambda c_\lambda\mathfrak{st}_\lambda$ in the star basis for any simple connected graph. In Section \ref{sec:2con}, we show how a single star coefficient characterizes whether $G$ is $2$-connected. 
\begin{theorem}
Let $G$ be a connected graph with $n\geq 3$ vertices and let $X_G=\sum_\lambda c_\lambda\mathfrak{st}_\lambda$. Then the coefficient $c_{21^{n-2}}$ is nonzero if and only if $G$ is $2$-connected.
\end{theorem}
This theorem is especially intriguing because we have an infinite family of examples demonstrating that $k$-connectivity for $k\geq 3$ cannot be determined from $X_G$.

In Section \ref{sec:linear}, we prove new linear relations on other star coefficients $c_\lambda$. For instance, we find formulas for star coefficients corresponding to partitions of the same length.  These formulas were used in \cite{chromunicyclic} to find the sizes of cycles in bicyclic graphs.  In general, the sums studied in Section \ref{sec:linear} are useful to uncover structural properties of $G$. In this paper, these sums allow us to find new bases for the vector spaces
\begin{equation*}
\mathcal C_n=\text{span}_{\mathbb Q}\{X_G: \ G\text{ connected }n\text{-vertex graph}\}\text{ and }\mathcal T_n=\text{span}_{\mathbb Q}\{X_T: \ T \ n\text{-vertex tree}\}.
\end{equation*}
In Corollary \ref{cor:catbasis} and Corollary \ref{cor:catcutbasis}, we give concise characterizations of these spaces in terms of the star basis coefficients. We also find a linear algebraic significance of the star graph $\text{St}_n$.

\begin{theorem}
For $n\geq 4$, the star graph $\text{St}_n$ is the unique connected $n$-vertex graph whose chromatic symmetric function cannot be expressed as a linear combination of the chromatic symmetric functions of other connected $n$-vertex graphs. 
\end{theorem}

Finally, in Section \ref{sec:coefn}, we relate the coefficient $c_n$ to acyclic orientations of $G$. As an application, we characterize the families of connected graphs $\mathcal F$ for which every chromatic symmetric function $X_G$ can be written as an integer linear combination in the corresponding chromatic basis $B_{\mathcal F}$ of Cho and van Willigenburg \cite{chrombases}.

\section{Background}

Let $G=(V,E)$ be a simple graph with $n$ vertices, meaning that $G$ has no multiple edges or loops. A \emph{proper colouring} of $G$ is a function $c:V\to\{1,2,3,\ldots\}$ such that $c(u)\neq c(v)$ whenever $\{u,v\}\in E$. The \emph{chromatic symmetric function} of $G$ is the formal power series in variables $\bm x=(x_1,x_2,x_3,\ldots)$ given by
\begin{equation*}
X_G (\bm{x})=\sum_{c\text{ proper }}\prod_{v\in V}x_{\kappa(v)}.
\end{equation*}
Hereafter, we will omit the $\bm x$ for brevity and simply write $X_G$.
The space $\Lambda^n$ of symmetric functions of degree $n$ has several bases indexed by \emph{integer partitions} of $n$, which are decreasing sequences of positive integers $\lambda=(\lambda_1,\ldots,\lambda_\ell)$ with $\lambda_1+\cdots+\lambda_\ell=n$. Some standard bases are the \emph{monomial symmetric functions} $m_\lambda$, which is the sum with coefficient one of all monomials $x_{i_1}^{t_1}\cdots x_{i_\ell}^{t_\ell}$, where the $t_i$ are the parts of $\lambda$ in some order; and the \emph{elementary symmetric functions} $e_\lambda$ defined by 
\begin{equation*}
e_\lambda=e_{\lambda_1}\cdots e_{\lambda_\ell},\text{ where }e_k=\sum_{i_1<\cdots<i_k}x_{i_1}\cdots x_{i_k}.
\end{equation*}
We will be particularly interested in the \emph{star basis}, defined as follows. For an integer $k$, the \emph{star graph} $\text{St}_k$ is the $k$-vertex tree where $(k-1)$ leaves are joined to a single central vertex. For an integer partition $\lambda$ of $n$ (denoted $\lambda\vdash n$), the \emph{star forest} $\text{St}_\lambda$ is the disjoint union of the stars $\text{St}_{\lambda_1},\ldots,\text{St}_{\lambda_\ell}$. Let $\mathfrak{st}_\lambda=X_{\text{St}_\lambda}$. Now a result of Cho and van Willigenburg \cite[Theorem 5]{chrombases} shows that the \emph{star basis} $\{\mathfrak{st}_\lambda: \ \lambda\vdash n\}$ is indeed a basis for $\Lambda^n$. \\

Aliste-Prieto, de Mier, Orellana, and Zamora \cite{markedgraphs} found an efficient way to calculate the star expansion of $X_G$ using a linear relation called the \emph{deletion-near-contraction (DNC) relation}. Given an edge $e=\{u,v\}\in E$ of $G$, we define the following graphs. \begin{itemize}
\item The \emph{deletion} $G\setminus e$ is given by deleting the edge $e$.
\item The \emph{contraction} $G/e$ is given by replacing the vertices $u$ and $v$ by a single vertex $uv$ and joining it to all vertices that are neighbours of $u$ or $v$.
\item The \emph{dot-contraction} $(G/e)^\circ$ is given by adding an isolated vertex $\circ$ to $G/e$.
\item The \emph{leaf-contraction} $(G/e)^\multimap$ is given by taking the dot-contraction $(G/e)^\circ$ and adding an edge joining the contracted vertex $uv$ to the isolated vertex $\circ$.
\end{itemize}

When doing contractions on graphs with cycles, multiple edges might arise after contracting an edge. We will replace multiple edges with a single edge. 

Now we have the following linear relation.

\begin{theorem}
\cite[Proposition 3.1]{markedgraphs} Let $e$ be an edge of $G$. Then $X_G$ satisfies the relation
\begin{equation} \label{eq:dncrelation}
X_G=X_{G\setminus e}-X_{(G/e)^\circ}+X_{(G/e)^\multimap}.
\end{equation}
\end{theorem}

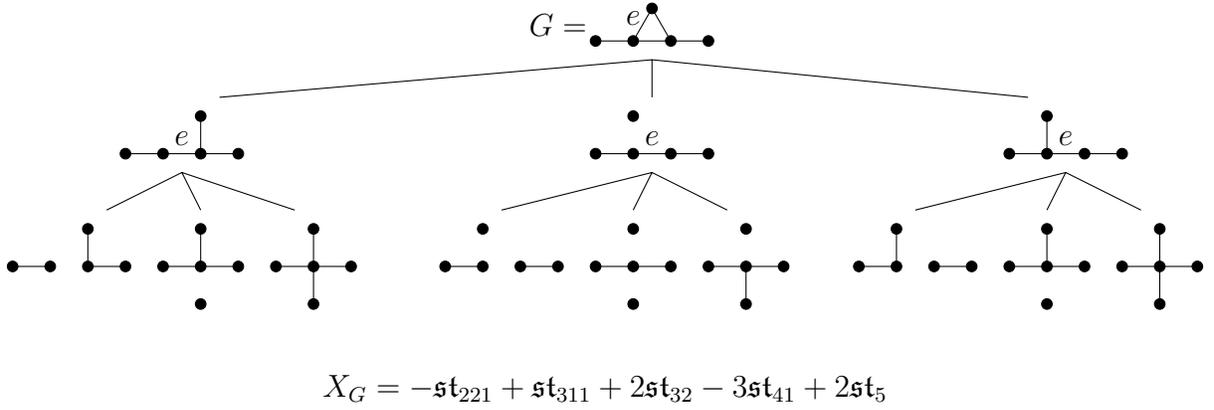
\begin{figure}
\begin{tikzpicture}
\node at (-0.5,0.216) () {$G=$};
\filldraw (0,0) circle (2pt) node[align=center,above] (){};
\filldraw (0.5,0) circle (2pt) node[align=center,above] (){};
\filldraw (0.75,0.433) circle (2pt) node[align=center,above] (){};
\filldraw (1,0) circle (2pt) node[align=center,above] (){};
\filldraw (1.5,0) circle (2pt) node[align=center,above] (){};
\draw (0,0)--(1.5,0) (0.5,0)--(0.75,0.433)--(1,0);
\draw (0.75,-0.25)--(-5,-0.75) (0.75,-0.25)--(0.75,-0.75) (0.75,-0.25)--(5.75,-0.75);
\node at (0.5,0.3) {$e$};

\filldraw (-6.25,-1.5) circle (2pt) node[align=center,above] (){};
\filldraw (-5.75,-1.5) circle (2pt) node[align=center,above] (){};
\filldraw (-5.25,-1) circle (2pt) node[align=center,above] (){};
\filldraw (-5.25,-1.5) circle (2pt) node[align=center,above] (){};
\filldraw (-4.75,-1.5) circle (2pt) node[align=center,above] (){};
\draw (-6.25,-1.5)--(-4.75,-1.5) (-5.25,-1)--(-5.25,-1.5);
\draw (-5.5,-1.75)--(-6.5,-2.25) (-5.5,-1.75)--(-5.25,-2.25) (-5.5,-1.75)--(-4,-2.25);
\node at (-5.5,-1.3) {$e$};

\filldraw (0,-1.5) circle (2pt) node[align=center,above] (){};
\filldraw (0.5,-1.5) circle (2pt) node[align=center,above] (){};
\filldraw (1,-1.5) circle (2pt) node[align=center,above] (){};
\filldraw (1.5,-1.5) circle (2pt) node[align=center,above] (){};
\filldraw (0.5,-1) circle (2pt) node[align=center,above] (){};
\draw (0,-1.5)--(1.5,-1.5);
\draw (0.75,-1.75)--(-1.25,-2.25) (0.75,-1.75)--(0.5,-2.25) (0.75,-1.75)--(1.75,-2.25);
\node at (0.75,-1.3) {$e$};

\filldraw (5.5,-1.5) circle (2pt) node[align=center,above] (){};
\filldraw (6,-1.5) circle (2pt) node[align=center,above] (){};
\filldraw (6,-1) circle (2pt) node[align=center,above] (){};
\filldraw (6.5,-1.5) circle (2pt) node[align=center,above] (){};
\filldraw (7,-1.5) circle (2pt) node[align=center,above] (){};
\draw (5.5,-1.5)--(7,-1.5) (6,-1.5)--(6,-1);
\draw (6.25,-1.75)--(4.25,-2.25) (6.25,-1.75)--(6,-2.25) (6.25,-1.75)--(7.25,-2.25);
\node at (6.25,-1.3) {$e$};

\filldraw (-7.75,-3) circle (2pt) node[align=center,above] (){};
\filldraw (-7.25,-3) circle (2pt) node[align=center,above] (){};
\filldraw (-6.75,-2.5) circle (2pt) node[align=center,above] (){};
\filldraw (-6.75,-3) circle (2pt) node[align=center,above] (){};
\filldraw (-6.25,-3) circle (2pt) node[align=center,above] (){};
\draw (-7.75,-3)--(-7.25,-3) (-6.75,-2.5)--(-6.75,-3)--(-6.25,-3);

\filldraw (-5.75,-3) circle (2pt) node[align=center,above] (){};
\filldraw (-5.25,-3.5) circle (2pt) node[align=center,above] (){};
\filldraw (-5.25,-2.5) circle (2pt) node[align=center,above] (){};
\filldraw (-5.25,-3) circle (2pt) node[align=center,above] (){};
\filldraw (-4.75,-3) circle (2pt) node[align=center,above] (){};
\draw (-5.75,-3)--(-4.75,-3) (-5.25,-2.5)--(-5.25,-3);

\filldraw (-3.75,-3.5) circle (2pt) node[align=center,above] (){};
\filldraw (-4.25,-3) circle (2pt) node[align=center,above] (){};
\filldraw (-3.75,-2.5) circle (2pt) node[align=center,above] (){};
\filldraw (-3.75,-3) circle (2pt) node[align=center,above] (){};
\filldraw (-3.25,-3) circle (2pt) node[align=center,above] (){};
\draw (-4.25,-3)--(-3.25,-3) (-3.75,-3.5)--(-3.75,-2.5);

\filldraw (-2,-3) circle (2pt) node[align=center,above] (){};
\filldraw (-1.5,-3) circle (2pt) node[align=center,above] (){};
\filldraw (-1.5,-2.5) circle (2pt) node[align=center,above] (){};
\filldraw (-1,-3) circle (2pt) node[align=center,above] (){};
\filldraw (-0.5,-3) circle (2pt) node[align=center,above] (){};
\draw (-2,-3)--(-1.5,-3) (-1,-3)--(-0.5,-3);

\filldraw (0,-3) circle (2pt) node[align=center,above] (){};
\filldraw (0.5,-3) circle (2pt) node[align=center,above] (){};
\filldraw (0.5,-2.5) circle (2pt) node[align=center,above] (){};
\filldraw (1,-3) circle (2pt) node[align=center,above] (){};
\filldraw (0.5,-3.5) circle (2pt) node[align=center,above] (){};
\draw (0,-3)--(1,-3);

\filldraw (1.5,-3) circle (2pt) node[align=center,above] (){};
\filldraw (2,-3) circle (2pt) node[align=center,above] (){};
\filldraw (2,-2.5) circle (2pt) node[align=center,above] (){};
\filldraw (2.5,-3) circle (2pt) node[align=center,above] (){};
\filldraw (2,-3.5) circle (2pt) node[align=center,above] (){};
\draw (1.5,-3)--(2.5,-3) (2,-3.5)--(2,-3);

\filldraw (3.5,-3) circle (2pt) node[align=center,above] (){};
\filldraw (4,-3) circle (2pt) node[align=center,above] (){};
\filldraw (4,-2.5) circle (2pt) node[align=center,above] (){};
\filldraw (4.5,-3) circle (2pt) node[align=center,above] (){};
\filldraw (5,-3) circle (2pt) node[align=center,above] (){};
\draw (3.5,-3)--(4,-3)--(4,-2.5) (4.5,-3)--(5,-3);

\filldraw (5.5,-3) circle (2pt) node[align=center,above] (){};
\filldraw (6,-3) circle (2pt) node[align=center,above] (){};
\filldraw (6,-2.5) circle (2pt) node[align=center,above] (){};
\filldraw (6.5,-3) circle (2pt) node[align=center,above] (){};
\filldraw (6,-3.5) circle (2pt) node[align=center,above] (){};
\draw (5.5,-3)--(6.5,-3) (6,-3)--(6,-2.5);

\filldraw (7,-3) circle (2pt) node[align=center,above] (){};
\filldraw (7.5,-3) circle (2pt) node[align=center,above] (){};
\filldraw (7.5,-2.5) circle (2pt) node[align=center,above] (){};
\filldraw (8,-3) circle (2pt) node[align=center,above] (){};
\filldraw (7.5,-3.5) circle (2pt) node[align=center,above] (){};
\draw (7,-3)--(8,-3) (7.5,-2.5)--(7.5,-3.5);

\end{tikzpicture}

\begin{equation*}
X_G=-\mathfrak{st}_{221}+\mathfrak{st}_{311}+2\mathfrak{st}_{32}-3\mathfrak{st}_{41}+2\mathfrak{st}_5
\end{equation*}
\caption{\label{fig:dnctree} A deletion-near-contraction tree for a graph $G$}
\end{figure} 

Note that if $e$ is a \emph{leaf edge} of $G$, meaning that one of its endpoints is a leaf, then we have $G\setminus e\cong(G/e)^\circ$ and $(G/e)^\multimap\cong G$, so the relation \eqref{eq:dncrelation} provides no new information. On the other hand, if $e$ is an \emph{internal edge} of $G$, meaning that it is not a leaf edge, then the graphs $G\setminus e$, $(G/e)^\circ$, and $(G/e)^\multimap$ all have fewer internal edges than $G$. Therefore, by repeatedly applying \eqref{eq:dncrelation}, we can express $X_G$ as a linear combination of the chromatic symmetric functions of graphs without internal edges, which are precisely the star forests. \\

We can keep track of this expansion by drawing a \emph{deletion-near-contraction (DNC) tree} for $G$, which is the following ternary rooted tree. We start with a single node, the graph $G$. Then whenever a leaf of our DNC tree is a graph $G'$ that has an internal edge, we choose some internal edge $e$ and assign three child nodes $G'\setminus e$, $(G'/e)^\circ$, and $(G'/e)^\multimap$ to $G'$. We repeat until all leaves of our tree are star forests. Then $X_G$ is the sum of these $\mathfrak{st}_\lambda$ with signs according to the number of dot-contractions used. An example is given in Figure \ref{fig:dnctree}.\\

Note that since we never delete a leaf edge, isolated vertices appear precisely when we use a dot-contraction. Therefore, if $G$ is connected, then every path in our DNC tree from $G$ to $\text{St}_\lambda$ uses $a_1$ dot-contractions, where $a_1$ is the number of parts of $\lambda$ equal to $1$. In particular, every leaf $\text{St}_\lambda$ of our DNC tree contributes $(-1)^{a_1}\mathfrak{st}_\lambda$ to $X_G$ and there is no cancellation.\\

We also remark that while a different sequence of choices of internal edges can result in a different DNC tree, they will produce the same star expansion because the set $\{\mathfrak{st}_\lambda: \ \lambda\vdash n\}$ forms a basis for $\Lambda^n$ \cite[Theorem 5]{chrombases}.  

\section{The coefficient of $\mathfrak{st}_{21^{n-2}}$ and $2$-connectivity}\label{sec:2con}

In this section, we prove that the coefficient $c_{21^{n-2}}$ in $X_G$ characterizes whether a graph $G$ is $2$-connected.

\begin{definition}
A graph $G=(V,E)$ is \emph{2-connected} if $G$ has at least three vertices and for every vertex $v\in V$, the graph $G\setminus v$, given by deleting $v$ and all incident edges, is connected.
\end{definition}

\begin{lemma}\label{lem:2con}
Let $G=(V,E)$ be a $2$-connected graph with at least four vertices and take any edge $e=\{u,v\}\in E$. Then $G\setminus e$ or $G/e$ is $2$-connected. 
\end{lemma}

\begin{proof}
Suppose that $G/e$ is not $2$-connected. This means that deleting some vertex $w$ of $G/e$ will disconnect the graph. Because $G$ is $2$-connected, if this vertex $w$ is not the contracted vertex $uv$, then $G\setminus w$ is connected and $(G/e)\setminus w=(G\setminus w)/e$ would be connected as well, so we must have $w=uv$. Let $x$ and $y$ be vertices not joined by a path in $(G/e)\setminus w=(G\setminus u)\setminus v$. Because $G$ is $2$-connected, there are paths $P_{xu}$ from $x$ to $u$ and $P_{uy}$ from $u$ to $y$ in $G\setminus v$, and there are paths $P_{xv}$ from $x$ to $v$ and $P_{vy}$ from $v$ to $y$ in $G\setminus u$. Because there is no path from $x$ to $y$ in $(G\setminus u)\setminus v$, the pairs $(P_{xu},P_{uy})$, $(P_{xu},P_{vy})$, $(P_{xv},P_{uy})$, and $(P_{xv},P_{vy})$ are internally vertex-disjoint. We now claim that $G\setminus e$ is $2$-connected.

Because $G$ is $2$-connected, the only way for $G\setminus e$ to not be $2$-connected is if deleting some vertex $w$ of $G\setminus e$ disconnects $u$ and $v$. But $G\setminus e$ has two internally vertex-disjoint walks from $u$ to $v$, namely the walk using $P_{xu}$ and $P_{xv}$, and the walk using $P_{uy}$ and $P_{vy}$. Therefore there will still be a path from $u$ to $v$ in any $(G\setminus e)\setminus w$, so indeed $G\setminus e$ is $2$-connected.
\end{proof}

\begin{theorem}\label{thm:2con}
Let $G$ be a connected graph with $n\geq 3$ vertices and let $X_G=\sum_\lambda c_\lambda\mathfrak{st}_\lambda$. Then the coefficient $c_{21^{n-2}}\neq 0$ if and only if $G$ is $2$-connected.
\end{theorem}

\begin{proof}
The coefficient $c_{21^{n-2}}$ is nonzero if and only if there is a path in a DNC tree from $G$ to the star forest $\text{St}_{21^{n-2}}$. If $G$ is $2$-connected, Lemma \ref{lem:2con} tells us that there is a sequence of deletions and dot-contractions that result in a graph consisting of $(n-3)$ isolated vertices and a $2$-connected three-vertex graph, which can only be the complete graph $K_3$. Applying one more dot contraction results in the star forest $\text{St}_{21^{n-2}}$.

Now suppose that $G$ is not $2$-connected. We use induction on $n$ to show that there is no path in a DNC tree from $G$ to $\text{St}_{21^{n-2}}$. If $n=3$, then $G$ must be the star $\text{St}_3$ and the DNC tree consists of only the node $G$, so suppose that $n\geq 4$. If $G$ has no internal edges, then $G=\text{St}_n$ and again we are done. Now suppose that $G$ has an internal edge $e$, then $G\setminus e$, $(G/e)^\circ$, and $(G/e)^\multimap$ are children of $G$ in a DNC tree. The graph $(G/e)^\multimap$ is connected and not $2$-connected, so by our induction hypothesis, there is no path from $(G/e)^\multimap$ to $\text{St}_{21^{n-2}}$. Similarly, a path from $(G/e)^\circ$ to $\text{St}_{21^{n-2}}$ would correspond to a path from $G/e$ to $\text{St}_{21^{n-3}}$ in a DNC tree for $G/e$, but $G/e$ is connected and not $2$-connected, so by our induction hypothesis no such path exists. If $G\setminus e$ is connected, then the same induction argument applies. If $G\setminus e$ is not connected, then because $e$ is an internal edge of $G$, the graph $G\setminus e$ has two connected components of at least two vertices. But both components must result in a star $\text{St}_k$ with $k\geq 2$ because the DNC relation is never applied to leaf edges. Therefore, there is once again no path from $G\setminus e$ to $\text{St}_{21^{n-2}}$. So $G$ has no path to $\text{St}_{21^{n-2}}$ and $c_{21^{n-2}}=0$.
\end{proof}

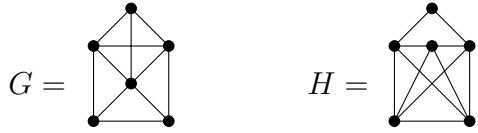
\begin{figure}
\begin{tikzpicture}

\node at (-0.75,0.5) {$G=$};
\filldraw (0,0) circle (2pt) node[align=center,above] (){};
\filldraw (0,1) circle (2pt) node[align=center,above] (){};
\filldraw (1,0) circle (2pt) node[align=center,above] (){};
\filldraw (1,1) circle (2pt) node[align=center,above] (){};
\filldraw (0.5,0.5) circle (2pt) node[align=center,above] (){};
\filldraw (0.5,1.5) circle (2pt) node[align=center,above] (){};
\draw (0,0)--(0,1)--(1,1)--(1,0)--(0,0) (0,1)--(0.5,1.5)--(1,1)--(0.5,0.5)--(0,1) (0.5,1.5)--(0.5,0.5) (0,0)--(0.5,0.5)--(1,0);

\node at (3.25,0.5) {$H=$};

\filldraw (4,0) circle (2pt) node[align=center,above] (){};
\filldraw (4,1) circle (2pt) node[align=center,above] (){};
\filldraw (5,0) circle (2pt) node[align=center,above] (){};
\filldraw (5,1) circle (2pt) node[align=center,above] (){};
\filldraw (4.5,1) circle (2pt) node[align=center,above] (){};
\filldraw (4.5,1.5) circle (2pt) node[align=center,above] (){};
\draw (4,0)--(4,1)--(5,1)--(5,0)--(4,0) (4,0)--(5,1)--(4.5,1.5)--(4,1)--(5,0) (4,0)--(4.5,1)--(5,0);

\end{tikzpicture}
\caption{\label{fig:kconexample} Graphs $G$ and $H$ with $X_G=X_H$, but $\kappa(G)=3$ and $\kappa(H)=2$}
\end{figure}

\begin{remark}
It turns out that a connected graph $G$ with at least three vertices is $2$-connected if and only if $(t-1)^2$ does not divide its chromatic polynomial $\chi_G(t)$ \cite[Theorem 1]{cutchrom}. So Theorem \ref{thm:2con} shows the curious fact that $c_{21^{n-2}}\neq 0$ if and only if $\chi'_G(1)\neq 0$ (here the prime indicates the derivative). 
\end{remark}

We now show that the chromatic symmetric function cannot detect $k$-connectivity for any $k\geq 3$. 

\begin{definition}
A graph $G=(V,E)$ is \emph{$k$-connected} if $G$ has at least $(k+1)$ vertices and for every set of $(k-1)$ vertices $\{v_1,\ldots,v_{k-1}\}\subset V$, the graph $G\setminus\{v_1,\ldots,v_{k-1}\}$, given by deleting all of the $v_i$ and all incident edges, is connected. The \emph{connectivity} $\kappa(G)$ is the largest $k$ such that $G$ is $k$-connected.
\end{definition}

\begin{example}
Consider the graphs $G$ and $H$ in Figure \ref{fig:kconexample}. By using Sage, or by enumerating proper colourings, we can calculate that
\begin{equation*}
X_G=X_H=12m_{2211}+96m_{21111}+720m_{111111}.
\end{equation*}
Furthermore, the graph $G$ is $3$-connected, whereas $H$ is only $2$-connected because the top vertex has degree two, so we can disconnect $H$ by removing its two neighbours. Therefore, the chromatic symmetric function cannot distinguish whether a graph is $3$-connected.
\end{example}

We can extend this example to $k$-connectivity for all $k\geq 3$ by adding \emph{universal vertices}, which are vertices adjacent to all others.

\begin{lemma}\label{lem:universal}
Let $G=(V,E)$ be a graph and let $X_G=\sum_{\lambda\vdash n} b_\lambda m_\lambda$ be its $m$-expansion. Let $G^\bullet$ be the graph obtained by adding a new vertex $\bullet$ and joining it to all vertices of $G$. Then $\kappa(G^\bullet)=\kappa(G)+1$ and 
\begin{equation*}
X_{G^\bullet}=\sum_{\mu\vdash n+1} a_1 b_{\mu\setminus 1}m_\mu,
\end{equation*}
where $a_1$ is the number of $1$'s in $\mu$ and $\mu\setminus 1$ is the partition obtained by removing a $1$.
\end{lemma}

\begin{proof}
Because the vertex $\bullet$ is adjacent to all other vertices of $G^\bullet$, it must be removed, along with $\kappa(G)$ vertices of $G$, in order to disconnect $G^\bullet$. The coefficient of $m_\mu$ in $X_{G^\bullet}$ counts the number of proper colourings $c$ using $\mu$ $1$'s, $\mu$ $2$'s, and so on. Because the vertex $\bullet$ is adjacent to all other vertices, it must be coloured using one of the $a_1$ colours $i$ for which $\mu_i=1$. Then the rest of the graph must be coloured according to the partition $\lambda=\mu\setminus 1$.
\end{proof}

\begin{corollary}
For every $k\geq 3$, there are $(k+3)$-vertex graphs $G$ and $H$ with $X_G=X_H$, but with $\kappa(G)=k$ and $\kappa(H)=k-1$.
\end{corollary}

\begin{proof}
We use induction on $k$. If $k=3$, we can take the graphs $G$ and $H$ in Figure \ref{fig:kconexample}. If $k\geq 4$, then by induction there are $(k+2)$-vertex graphs $G'$ and $H'$ with $X_{G'}=X_{H'}$, $\kappa(G')=k-1$, and $\kappa(H')=k-2$. By Lemma \ref{lem:universal}, the $(k+3)$-vertex graphs $G=(G')^\bullet$ and $H=(H')^\bullet$ now satisfy $X_G=X_H$, $\kappa(G)=k$, and $\kappa(H)=k-1$.
\end{proof}

We can ask whether there are $n$-vertex graphs $G$ and $H$ with $X_G=X_H$ and connectivity differing by more than $1$. We checked by computer that there are no such examples for $n\leq 8$. This suggests that the chromatic symmetric function $X_G$ could encode the ``approximate connectivity'' of $G$. 

\section{Bases for spans of chromatic symmetric functions}\label{sec:linear}

In this section, we prove new linear relations on star coefficients. This will give us new insight into the linear algebra of chromatic symmetric functions.

\begin{lemma}\label{lem:hookrelation}
Let $G$ be a connected $n$-vertex graph with $n\geq 3$ and let $X_G=\sum_\lambda c_\lambda\mathfrak{st}_\lambda$. Consider the sum of hook star coefficients
\begin{equation}
h(G)=2c_{21^{n-2}}+c_{31^{n-3}}+c_{41^{n-4}}+\cdots+c_{(n-2)11}+c_{(n-1)1}+c_n.
\end{equation}
Then we have
\begin{equation}\label{eq:hookrelation}
h(G)=\begin{cases}0,&\text{ if }G\neq \text{St}_n,\\
1,&\text{ if }G=\text{St}_n.
\end{cases}
\end{equation}
\end{lemma}

\begin{proof}
We use induction on $n$. If $n=3$, the only possibilities for $G$ are the star $\text{St}_3$ and the complete graph $K_3$, and it is routine to check that \eqref{eq:hookrelation} holds, so suppose that $n\geq 4$. We also use induction on the number of internal edges of $G$. If $G$ has no internal edges, then $G=\text{St}_n$, $c_n=1$, and all the other $c_\lambda=0$, so $h(\text{St}_n)=1$. Now suppose that $G$ has some internal edge $e=\{u,v\}$. We will show that $h(G)=0$. By the DNC relation, we have
\begin{equation}\label{eq:hookinduction}
h(G)=h(G\setminus e)-h((G/e)^\circ)+h((G/e)^\multimap)=h(G\setminus e)-h(G/e)+h((G/e)^\multimap).
\end{equation}
Note that $h((G/e)^\circ)=h(G/e)$ because the star expansion of $(G/e)^\circ$ is obtained from that of $G/e$ by appending a $1$ to each partition. Our induction hypothesis allows us to calculate the right hand side of \eqref{eq:hookinduction}. We consider different cases, depending on whether $G\setminus e=\text{St}_n$, $G/e=\text{St}_{n-1}$, and $(G/e)^\multimap=\text{St}_n$. If none of these hold, then $h(G)=0-0+0=0$. Now suppose that $G/e=\text{St}_{n-1}$. If the contracted vertex $uv$ is the central vertex of $\text{St}_{n-1}$, then because $e$ was an internal edge of $G$, both $u$ and $v$ have other neighbours in $G$, so we have $G\setminus e\neq\text{St}_n$. We also have $(G/e)^\multimap=\text{St}_n$, so $h(G)=0-1+1=0$. If the contracted vertex $uv$ is not the central vertex $w$ of $\text{St}_{n-1}$, then $u$ and $v$ must have been neighbours of $w$, so we have $G\setminus e=\text{St}_n$, $(G/e)^\multimap\neq\text{St}_n$, and $h(G)=1-1+0=0$. 

If $G\setminus e=\text{St}_n$, then $G$ must have been the $n$-vertex star with an additional edge joining two leaves, so $G/e=\text{St}_{n-1}$ and we have already seen this case. Similarly, if $(G/e)^\multimap=\text{St}_n$, then $G/e=\text{St}_{n-1}$ and again we have seen this case. Therefore, $h(G)=0$ in all cases. 
\end{proof}

\begin{corollary}\label{cor:star}
The chromatic symmetric function $\mathfrak{st}_n$ does not lie in the space
\begin{equation}
\text{span}_{\mathbb Q}\{X_G: \ G\text{ connected }n\text{-vertex graph, }G\neq\text{St}_n\}.
\end{equation}
\end{corollary}

\begin{proof}
By Lemma \ref{lem:hookrelation}, the star $\text{St}_n$ is the unique connected $n$-vertex graph that does not satisfy the linear equation $h(G)=0$.
\end{proof}

We will show in Theorem \ref{thm:starcoloop} that for $n\geq 4$, the star $\text{St}_n$ is the \emph{only} connected $n$-vertex graph whose chromatic symmetric function cannot be expressed as a linear combination of the chromatic symmetric functions of other graphs.\\

We now turn our attention to trees. Gonzalez, Orellana, and Tomba identifed the \emph{leading partition} of $X_T=\sum_\lambda c_\lambda\mathfrak{st}_\lambda$, meaning the lexicographically smallest partition $\mu$ with $c_\mu\neq 0$.

\begin{theorem} \label{thm:leading} \cite[Theorem 4.15]{chromstar}
Let $T$ be a tree and let $F$ be the forest obtained by deleting all internal edges of $T$. Then the leading partition of $X_T$ is the partition whose parts are the sizes of the connected components of $F$.
\end{theorem}

They then showed that every possible leading partition is attained by some \emph{caterpillar}.

\begin{definition}
Let $\alpha=\alpha_1\cdots\alpha_\ell$ be a \emph{composition} of size $n$ (denoted $\alpha\vDash n)$, meaning a sequence of positive integers with sum $n$. Assume that $\alpha_1,\alpha_\ell\geq 2$. The \emph{caterpillar} $\text{Cat}_\alpha$ is the $n$-vertex tree formed by taking a path of $\ell$ vertices $v_1,\ldots,v_\ell$ and attaching $(\alpha_i-1)$ leaves to vertex $v_i$ for $1\leq i\leq\ell$. 
\end{definition}

\begin{corollary}\label{cor:catleading}
The leading partition of $X_{\text{Cat}_\alpha}$ is $\lambda=\text{sort}(\alpha)$, given by writing $\alpha$ in decreasing order.
\end{corollary}

\begin{corollary}
A partition $\lambda$ is the leading partition of $X_T$ for some $n$-vertex tree if and only if either $\lambda=n$ or $\lambda_2\geq 2$.
\end{corollary}

\begin{proof}
The leading partition of $\mathfrak{st}_n$ is $n$ and if $\lambda_2\geq 2$, then $\lambda$ is the leading partition of $X_{\text{Cat}_{\lambda_2\lambda_3\cdots\lambda_\ell\lambda_1}}$. Conversely, if $T$ is a tree other than $\text{St}_n$, then it has at least one internal edge $e$ and both components of $T\setminus e$ have a leaf of $T$. When we delete the other internal edges of $T$, those leaves will be in a connected component of size at least $2$, so the leading partition $\lambda$ must satisfy $\lambda_2\geq 2$. 
\end{proof}

We can use leading partitions to show that the chromatic symmetric function of the path $P_n$ is not a linear combination of the chromatic symmetric functions of other trees.

\begin{corollary}\label{cor:path}
The chromatic symmetric function $X_{P_n}$ does not lie in the space
\begin{equation}
\text{span}_{\mathbb Q}\{X_T: \ T \ n\text{-vertex tree}, \ T\neq P_n\}.
\end{equation}
\end{corollary}

\begin{proof}
If $n=3$, the path $P_n$ is the only $n$-vertex tree, so suppose that $n\geq 4$. By Theorem \ref{thm:leading}, the leading partition for $P_n=\text{Cat}_{21^{n-4}2}$ is $\lambda=221^{n-4}$. All other trees $T$ have at least three leaves, so the forest obtained by deleting all internal edges has at least three connected components of size at least $2$. Therefore, the leading partition is lexicographically larger than $\lambda$ and $\mathfrak{st}_{221^{n-4}}$ does not appear in the star expansion of $X_T$.
\end{proof}

We now show a linear relation that is satisfied by almost all trees.

\begin{lemma}\label{lem:nearhookrelation}
Let $T$ be a connected $n$-vertex tree with $n\geq 5$ and let $X_T=\sum_\lambda c_\lambda\mathfrak{st}_\lambda$. Consider the sum of star coefficients
\begin{equation}
nh(T)=c_{221^{n-4}}+c_{321^{n-5}}+\cdots+c_{(n-3)21}+c_{(n-2)2}.
\end{equation}
Then we have
\begin{equation}\label{eq:nearhookrelation}
nh(T)=\begin{cases}
0,&\text{ if }T\neq \text{Cat}_{(n-2)2}, \ T\neq P_n,\\
1,&\text{ if }T=\text{Cat}_{(n-2)2},\\
(-1)^{n-1},&\text{ if }T=P_n.
\end{cases}
\end{equation}
\end{lemma}

\begin{proof}
We use induction on $n$. If $n=5$, the only possibilities for $T$ are the star $\text{St}_5$, the caterpillar $\text{Cat}_{32}$, and the path $P_5$, and it is routine to check that \eqref{eq:nearhookrelation} holds. We also use induction on the number of internal edges of $T$. If $T$ has no internal edges, then $T=\text{St}_n$, $c_n=1$, and all the other $c_\lambda=0$, so $nh(\text{St}_n)=0$. Now suppose that $T$ has some internal edge $e=\{u,v\}$. By the DNC relation, we have
\begin{equation}\label{eq:nearhookinduction}
nh(T)=nh(T\setminus e)-nh((T/e)^\circ)+nh((T/e)^\multimap)=nh(T\setminus e)-nh(T/e)+nh((T/e)^\multimap).
\end{equation}
Our induction hypothesis allows us to calculate the right hand side of \eqref{eq:nearhookinduction}. We consider different cases, depending on whether $T\setminus e$, $T/e$, and $(T/e)^\multimap$ are paths or caterpillars. 

We first consider $T\setminus e$. Because $T$ is a tree and $e$ is an internal edge, $T\setminus e$ consists of trees $T_1$ and $T_2$ with some sizes $n_1\geq n_2\geq 2$. Suppose that $n_2\geq 3$. Now the trees $T_1$ and $T_2$ have at least three vertices and are not $2$-connected, so by Theorem \ref{thm:2con}, their star coefficients $c_{21^{n_1-2}}(T_1)$ and $c_{21^{n_2-2}}(T_2)$ are both $0$. Therefore, the star coefficients $c_{a21^{n-a-2}}(T\setminus e)$ are all $0$ and $nh(T\setminus e)=0$. If $n_2=2$, then we have
\begin{equation}
nh(T\setminus e)=c_{21^{n-6}}(T_1)+c_{31^{n-5}}(T_1)+\cdots+c_{n-2}(T_1)=h(T_1)-c_{21^{n-6}}(T_1).
\end{equation}
Because $T_1$ has at least three vertices and is not $2$-connected, we have $c_{21^{n-6}}(T_1)=0$ by Theorem \ref{thm:2con}. By Lemma \ref{lem:hookrelation}, we have $h(T_1)=0$ unless $T_1=\text{St}_{n-2}$, which can only occur if $T=\text{Cat}_{(n-2)2}$, or if $T=\text{Cat}_{(n-3)12}$ and $e$ is the internal edge where $T\setminus e=\text{St}_{(n-2)2}$. Therefore
\begin{equation}
nh(T\setminus e)=\begin{cases}
1,&\text{ if }T=\text{Cat}_{(n-2)2},\\
1,&\text{ if }T=\text{Cat}_{(n-3)12}\text{ and }T\setminus e=\text{St}_{(n-2)2},\\
0,&\text{ otherwise.}
\end{cases}
\end{equation}
We next consider the trees $T/e$ and $(T/e)^\multimap$. We have $T/e=P_{n-1}$ if and only if $T=P_n$. We have $T/e=\text{Cat}_{(n-3)2}$ if and only if $T$ is either $\text{Cat}_{2(n-4)2}$ or $\text{Cat}_{(n-3)12}$. Because $e$ is an internal edge of $T$, the contracted vertex $uv$ of $(T/e)^\multimap$ has degree at least $3$ so $(T/e)^\multimap$ is never $P_n$. We have $(T/e)^\multimap=\text{Cat}_{(n-2)2}$ if and only if either $T=\text{Cat}_{2(n-4)2}$, or $T=\text{Cat}_{(n-3)12}$ and $e$ is the internal edge where $T\setminus e=\text{St}_{(n-3)3}$. Putting this together, we have
\begin{equation}
nh(T)=\begin{cases}
1-0+0=1,&\text{ if }T=\text{Cat}_{(n-2)2},\\
0-(-1)^{n-2}+0=(-1)^{n-1},&\text{ if }T=P_n,\\
1-1+0=0,&\text{ if }T=\text{Cat}_{(n-3)12}\text{ and }T\setminus e=\text{St}_{(n-2)2},\\
0-1+1=0,&\text{ if }T=\text{Cat}_{(n-3)12}\text{ and }T\setminus e=\text{St}_{(n-3)3},\\
0-1+1=0,&\text{ if }T=\text{Cat}_{2(n-4)2},\\
0-0+0=0,&\text{ otherwise,}
\end{cases}
\end{equation}
as desired.
\end{proof}

\begin{corollary}\label{cor:cat}
The chromatic symmetric function $X_{\text{Cat}_{(n-2)2}}$ does not lie in the space
\begin{equation}
\text{span}_{\mathbb Q}\{X_T: \ T \ n\text{-vertex tree}, T\neq \text{Cat}_{(n-2)2}\}.
\end{equation}
\end{corollary}

\begin{proof}
Suppose that there are coefficients $a,a_T\in\mathbb Q$ for which
\begin{equation}
X_{\text{Cat}_{(n-2)2}}=aX_{P_n}+\sum_{T\neq P_n}a_TX_T.
\end{equation}
By Corollary \ref{cor:path}, we cannot write the chromatic symmetric function of the path $P_n$ as a linear combination of the chromatic symmetric functions of other trees, so we must have $a=0$. But now by Lemma \ref{lem:nearhookrelation}, we have $nh(\text{Cat}_{(n-2)2})=1$, while the corresponding star coefficient sum on the right hand side is $0$.
\end{proof}

We have seen that the chromatic symmetric function of the star $\text{St}_n$, the path $P_n$, and the caterpillar $\text{Cat}_{(n-2)2}$ cannot be expressed as linear combinations of the chromatic symmetric functions of other trees. We will show in Theorem \ref{thm:treecoloops} that for $n\geq 8$, these are the \emph{only} trees with this property. We first find bases for the vector space
\begin{equation}
\mathcal T_n=\text{span}_{\mathbb Q}\{X_T: \ T \ n\text{-vertex tree}\}.
\end{equation}

Gonzalez, Orellana, and Tomba found an elegant description of $\mathcal T_n$ using the star basis.

\begin{lemma}\label{lem:catbasis}
\cite[Proposition 6.9]{chromstar} The set of chromatic symmetric functions of caterpillars
\begin{equation}\label{eq:catbasis}
B=\{X_{\text{Cat}_{\lambda_2\lambda_3\cdots\lambda_\ell\lambda_1}}: \ \lambda=\lambda_1\lambda_2\cdots\lambda_\ell\vdash n, \text{ either }\lambda=n\text{ or } \lambda_2\geq 2\}
\end{equation}
is linearly independent.
\end{lemma}

\begin{proof}
By Corollary \ref{cor:catleading}, these caterpillars have distinct leading partitions.
\end{proof}

They also proved the following relation on star coefficients with a given length.

\begin{definition}
Let $G$ be an $n$-vertex graph and let $X_G=\sum_\lambda c_\lambda\mathfrak{st}_\lambda$. Fix an integer $1\leq\ell\leq n$. We define $\sigma_\ell(G)$ to be the sum of star coefficients
\begin{equation}
\sigma_\ell(G)=\sum_{\lambda: \ \ell(\lambda)=\ell}c_\lambda.
\end{equation}
\end{definition}

\begin{lemma}\label{lem:lengthrelation}
\cite[Lemma 6.10]{chromstar} If $T$ is a tree and $\ell\geq 2$, then $\sigma_\ell(T)=0$.
\end{lemma}

\begin{corollary}\label{cor:catbasis}
We have that
\begin{equation}\label{eq:treespan}
\mathcal T_n=\left\{\sum_\lambda a_\lambda\mathfrak{st}_\lambda: \ a_\lambda\in\mathbb Q, \ \sum_{\ell(\lambda)=\ell}a_\lambda=0\text{ for every }2\leq\ell\leq n\right\}.
\end{equation}
Moreover, the set $B$ in \eqref{eq:catbasis} is a basis for $\mathcal T_n$. 
\end{corollary}

\begin{proof}
Let $\mathcal T_n'$ be the right hand side of \eqref{eq:treespan}. The inclusion $\mathcal T_n\subseteq\mathcal T_n'$ holds by Lemma \ref{lem:lengthrelation}. By Lemma \ref{lem:catbasis}, we have that $\dim\mathcal T_n\geq |B|=p(n)-n+1=\dim\mathcal T'_n$, where $p(n)$ is the number of integer partitions of size $n$. Therefore $\mathcal T_n=\mathcal T_n'$ and $B$ is a basis for $\mathcal T_n$.
\end{proof}

We will also need the following lemma for the caterpillars $\text{Cat}_{a1^kb}$ and $\text{Cat}_{aaa}$, which are tricky to handle because they are the only trees with their leading partition. The proof is a lengthy calculation and not too enlightening so we will postpone it to the end of this section.

\begin{lemma}\hspace{0pt}
\label{lem:ab1k}
\begin{enumerate}
\item Fix integers $a\geq b\geq 2$ and $k\geq 0$ with $a\geq 3$, $a+b+k\geq 8$, and either $b\geq 3$ or $k\geq 1$. The chromatic symmetric function of the caterpillar $X_{\text{Cat}_{a1^kb}}$ can be written as a linear combination of the chromatic symmetric functions of other trees. 
\item Fix an integer $a\geq 3$. The chromatic symmetric function of the caterpillar $X_{\text{Cat}_{aaa}}$ can be expressed as a linear combination of the chromatic symmetric functions of other trees.
\end{enumerate}
\end{lemma}

\begin{theorem}\label{thm:treecoloops}
Let $n\geq 8$ and let $T$ be an $n$-vertex tree other than $\text{St}_n$, $P_n$, and $\text{Cat}_{(n-2)2}$. Then $X_T$ can be written as a linear combination of the chromatic symmetric functions of other trees.
\end{theorem}

\begin{proof}
By Corollary \ref{cor:catbasis}, the caterpillars in \eqref{eq:catbasis} form a basis for $\mathcal T_n$, so we need only consider the case where $T$ is one of these caterpillars. Let $T=\text{Cat}_\alpha$ and $\lambda=\text{sort}(\alpha)$. Because $T\neq\text{St}_n$, we have $\lambda_2\geq 2$. Our strategy will be to find another tree $T'$ with the same leading partition $\lambda$. Then the set $B'=B\setminus\{X_T\}\cup\{X_{T'}\}$ will be linearly independent because the leading partitions are distinct. Because $|B'|=|B|=\dim\mathcal T_n$, the set $B'$ will be a basis for $\mathcal T_n$, and we can write $X_T$ as a linear combination of the chromatic symmetric functions of trees in $B'$, as desired. 

Let $t\geq 2$ be maximal with $\lambda_t\geq 2$, so that $\lambda$ is of the form $\lambda_1\cdots\lambda_t1^k$ for some $k\geq 0$. First suppose that $t\geq 4$ and $k=0$. In this case, we can take $T'$ to be the tree obtained from the star $\text{St}_t$ by attaching $(\lambda_t-1)$ leaves to the central vertex and attaching $(\lambda_i-1)$ leaves to the $i$-th leaf of $\text{St}_t$ for $1\leq i\leq t-1$. Then $X_{T'}$ has leading partition $\lambda$ and $T'$ is not a caterpillar because its non-leaves form the star $\text{St}_t$, not a single path. 

The same construction works if $t=3$ and $k\geq 1$. If $t\geq 4$ and $k\geq 1$, we can take $T'$ to be the tree obtained from the caterpillar $\text{Cat}_{(t-1)1^{k-1}2}$, attaching $(\lambda_i-1)$ leaves to the $i$-th leaf of the degree-$(t-1)$ vertex for $1\leq i\leq t-2$, and attaching $(\lambda_{t-1}-1)$ leaves to the other leaf of $\text{Cat}_{(t-1)1^{k-1}2}$. Again, $X_{T'}$ has leading partition $\lambda$ and $T'$ is not a caterpillar because its non-leaves form the caterpillar $\text{Cat}_{(t-1)1^{k-1}2}$, not a single path.

Now suppose that $t=3$ and $k=0$, so $T=\text{Cat}_{\lambda_2\lambda_3\lambda_1}$. If $\lambda_1\neq\lambda_3$, then the tree $T'=\text{Cat}_{\lambda_2\lambda_1\lambda_3}$ also has leading partition $\lambda$ and we are done as before. If $\lambda_1=\lambda_3$, then $T=\text{Cat}_{aaa}$ for some $a\geq 3$ and $X_T$ can be expressed as a linear combination of the chromatic symmetric functions of other trees by Part (2) of Lemma \ref{lem:ab1k}.

Finally, if $t=2$, then $T=\text{Cat}_{a1^kb}$ for some $a\geq b\geq 2$ and $k\geq 0$. Because $T\neq P_n$, we must have $a\geq 3$, and because $T\neq\text{Cat}_{(n-2)2}$, we have either $b\geq 3$ or $k\geq 1$. Therefore $X_T$ can be expressed as a linear combination of the chromatic symmetric functions of other trees by Part (1) of Lemma \ref{lem:ab1k}. 
\end{proof}

\begin{remark}
We can also use computer calculations to describe the cases of $n\leq 7$. When $n\leq 5$, the only chromatic symmetric functions of a tree that are not linear combinations of other chromatic symmetric functions of other trees are still $P_n$, $\text{Cat}_{(n-2)2}$, and $\text{St}_n$, although these can coincide (for example, $P_3=\text{St}_3$ and $P_4=\text{Cat}_{22}$). When $n=6$, the chromatic symmetric functions of all six trees are linearly independent because they have distinct leading partitions. For $n=7$, the chromatic symmetric functions of trees that cannot be expressed as linear combinations of the chromatic symmetric functions of other trees are $P_7$, $\text{St}_7$, $\text{Cat}_{52}$, and $\text{Cat}_{43}$.  
\end{remark}

We now return to the vector space spanned by all connected graphs
\begin{equation}
\mathcal C_n=\text{span}_{\mathbb Q}\{X_G: \ G\text{ connected }n\text{-vertex graph}\}.
\end{equation}

We can extend a basis of $\mathcal T_n$ to a basis of $\mathcal C_n$ by adding certain connected \emph{unicyclic} graphs, which are graphs with a single cycle. 

\begin{definition}
Let $c\geq 3$ and $\ell\geq 0$. The \emph{cuttlefish} $\text{Cut}_{c,\ell}$ is the $(c+\ell)$-vertex graph formed by taking a cycle of $c$ vertices and attaching $\ell$ leaves to a fixed vertex $v$.
\end{definition}

Gonzalez, Orellana, and Tomba also identified the leading partition for connected unicyclic graphs.

\begin{theorem}\cite[Theorem 5.6]{chromunicyclic} 
Let $G$ be a connected unicyclic graph that consists of a tree attached to a vertex $v$ of a cycle. Let $e$ be an edge of the cycle incident to $v$ and consider the tree $T=G\setminus e$. Then the leading partition of $X_T$ is the same as that of $X_G$.
\end{theorem}

\begin{corollary}\label{cor:cutleading}
The cuttlefish chromatic symmetric function $X_{\text{Cut}_{c,n-c}}$ has leading partition $\lambda=(n-c+1)21^{c-3}$.
\end{corollary}

\begin{lemma}\label{lem:cutbasis}
The set of chromatic symmetric functions of cuttlefish
\begin{equation}\label{eq:cutbasis}
B'=\{X_{\text{Cut}_{c,n-c}}: \ 3\leq c\leq n\}
\end{equation}
is linearly independent.
\end{lemma}

\begin{proof}
By Corollary \ref{cor:cutleading}, these cuttlefish have distinct leading partitions.
\end{proof}

In order to show that cuttlefish are independent from caterpillars, we will use the following relation for the sum of star coefficients in a given length.

\begin{lemma}\label{lem:unilength} Let $G$ be a connected unicyclic graph with a cycle of length $c$ and fix an integer $1\leq\ell\leq n$. Then we have
\begin{equation}
\sigma_\ell(G)=(-1)^{\ell-1}\binom{c-1}\ell.
\end{equation}
\end{lemma}

\begin{proof}
We use induction on $c\geq 3$ and the DNC relation. First suppose that $c=3$. If $T$ is a tree, then $\sigma_\ell(T)=0$ if $\ell\geq 2$ by Lemma \ref{lem:lengthrelation} and $\sigma_1(T)=c_n(T)=1$ by \cite[Section 3]{chromstar}. Let $e$ be an edge of the cycle of $G$ and note that $G\setminus e$, $G/e$, and $(G/e)^\multimap$ are trees. Then we have
\begin{align*}
\sigma_\ell(G)&=\sigma_\ell(G\setminus e)-\sigma_\ell((G/e)^\circ)+\sigma_\ell((G/e)^\multimap)=\sigma_\ell(G\setminus e)-\sigma_{\ell-1}(G/e)+\sigma_\ell((G/e)^\multimap)\\\nonumber&=\left(\begin{cases}1-0+1=2,&\text{ if }\ell=1,\\
0-1+0=-1,&\text{ if }\ell=2,\\
0,&\text{ if }\ell\geq 3,\end{cases}\right)=(-1)^{\ell-1}\binom{3-1}\ell.
\end{align*} Similarly, for $c\geq 4$, if $e$ is an edge of the cycle, then $G\setminus e$ is a tree, while $(G/e)$ and $(G/e)^\multimap$ are connected unicyclic graphs with a cycle of length $(c-1)$. Therefore, by induction, we have for $\ell\geq 2$ that
\begin{align*}
\sigma_\ell(G)&=\sigma_\ell(G\setminus e)-\sigma_\ell((G/e)^\circ)+\sigma_\ell((G/e)^\multimap)=\sigma_\ell(G\setminus e)-\sigma_{\ell-1}(G/e)+\sigma_\ell((G/e)^\multimap)\\\nonumber&=0-(-1)^{\ell-2}\binom{c-2}{\ell-1}+(-1)^{\ell-1}\binom{c-2}\ell=(-1)^{\ell-1}\binom{c-1}\ell.
\end{align*}
Similarly, for $\ell=1$, we have $\sigma_1(G)=\sigma_1(G\setminus e)+\sigma_1((G/e)^\multimap)=1+(c-2)=c-1$.
\end{proof}

\begin{remark}
This shows that a connected unicyclic graph $G$ with a cycle of length $c$ has $c_n=c-1$. Therefore, the size of the cycle can be read off from the star expansion of $G$.
\end{remark}

\begin{remark}
If $G$ is \emph{bicyclic}, meaning it has exactly two cycles, a similar relation on $\sigma_\ell(G)$ can be used to determine the cycle lengths from the star expansion of $X_G$ \cite[Theorem 7.6]{chromunicyclic}.
\end{remark}

\begin{corollary}\label{cor:catcutind}
The set of chromatic symmetric functions $B\cup B'$ is linearly independent, where $B$ is given in \eqref{eq:catbasis} and $B'$ is given in \eqref{eq:cutbasis}.
\end{corollary}

\begin{proof}
Suppose that we have a nontrivial linear dependence
\begin{equation}
\sum_{c=3}^na_cX_{\text{Cut}_{c,n-c}}+\sum_{\lambda: \ \lambda_2\geq 2}a_\lambda X_{\text{Cat}_{\lambda_2\lambda_3\cdots\lambda_\ell\lambda_1}}+a\mathfrak{st}_n=0.
\end{equation}
By Lemma \ref{lem:catbasis}, the caterpillars are linearly independent, so we must have $a_{c'}\neq 0$ for some maximal $3\leq c'\leq n$. But now by Lemma \ref{lem:lengthrelation} and Lemma \ref{lem:unilength}, we have $\sigma_{c'}(T)=0$ for every caterpillar $T$, $\sigma_{c'}(\text{Cut}_{c,n-c})=0$ for $c<c'$, and $\sigma_{c'}(\text{Cut}_{c',n-c'})\neq 0$. Therefore, by taking the sum of star coefficients of length $c'$ on each side, we get that $a_{c'}=0$, a contradiction. 
\end{proof}

\begin{corollary}\label{cor:catcutbasis}
Let $n\geq 2$. We have that
\begin{equation}\label{eq:connectedspan}
\mathcal C_n=\left\{\sum_\lambda a_\lambda\mathfrak{st}_\lambda: \ a_\lambda\in\mathbb Q, \ a_{1^n}=0\right\}.
\end{equation}
Moreover, the set $B\cup B'$ is a basis for $\mathcal C_n$, where $B$ is given in \eqref{eq:catbasis} and $B'$ is given in \eqref{eq:cutbasis}.
\end{corollary}

\begin{proof}
Let $\mathcal C'_n$ be the right hand side of \eqref{eq:connectedspan}. The inclusion $\mathcal C_n\subseteq\mathcal C'_n$ holds because when we use the DNC relation to calculate the star expansion of $X_G$, we only apply it to internal edges, so we can never get the star $\text{St}_{1^n}$ in a DNC tree. By Corollary \ref{cor:catcutind}, the set $B\cup B'$ is linearly independent, so we have that \begin{equation}\dim\mathcal C_n\geq |B|+|B'|=(p(n)-n+1)+(n-2)=p(n)-1=\dim\mathcal C'_n.\end{equation}
Therefore $\mathcal C_n=\mathcal C'_n$ and $B\cup B'$ is a basis for $\mathcal C_n$.
\end{proof}

The following lemma will give us some useful linear relations on cuttlefish.

\begin{lemma}\label{lem:cutrelations}\hspace{0pt}
\begin{enumerate}
\item Fix integers $3\leq c\leq n-1$ and let $G$ be the graph obtained from $\text{Cut}_{c,n-c-1}$ by attaching a leaf to a vertex in the cycle adjacent to the vertex of degree $(n-c+1)$. Then we have
\begin{equation}
X_{\text{Cut}_{c,n-c}}=X_{\text{Cat}_{(n-c+1)1^{c-3}2}}-X_{\text{Cat}_{(n-c)1^{c-2}2}}+X_G.
\end{equation}
\item Let $C_n$ be the cycle on $n\geq 4$ vertices and let $G$ be the graph obtained from $C_{n-1}$ by adding a new vertex and joining it to two adjacent vertices of $C_{n-1}$. Then we have
\begin{equation}
X_{C_n}=X_{P_n}+X_G-X_{\text{Cut}_{n-1,1}}.
\end{equation}
\item Let $G$ be the graph obtained from $\text{St}_n$ by adding two edges $e$ and $e'$ joining a leaf $v$ to two others. Then we have
\begin{equation}
X_{\text{Cat}_{(n-2)2}}=X_{\text{Cut}_{4,n-4}}-X_G+X_{(G/e)^\multimap}.
\end{equation}
\end{enumerate}
\end{lemma}

\begin{proof}\hspace{0pt}
\begin{enumerate}
\item Let $e$ be an edge of the cycle of $\text{Cut}_{c+1,n-c-1}$ incident to the vertex of degree $(n-c+2)$. By the DNC relation, we have
\begin{equation}\label{eq:cutrelation1}
X_{\text{Cut}_{c+1,n-c-1}}=X_{\text{Cat}_{(n-c)1^{c-2}2}}-X_{\text{Cut}_{c,n-c-1}}\mathfrak{st}_1+X_{\text{Cut}_{c,n-c}},
\end{equation}
and by applying the DNC relation to an edge $e'$ adjacent to $e$ in the cycle, we have
\begin{equation}\label{eq:cutrelation2}
X_{\text{Cut}_{c+1,n-c-1}}=X_{\text{Cat}_{(n-c+1)1^{c-3}2}}-X_{\text{Cut}_{c,n-c-1}}\mathfrak{st}_1+X_G.
\end{equation}
Now the result follows by subtracting \eqref{eq:cutrelation2} from \eqref{eq:cutrelation1}.
\item By applying the DNC relation to an edge $e$ of $C_n$, we have
\begin{equation}\label{eq:cutrelation3}
X_{C_n}=X_{P_n}-X_{C_{n-1}}\mathfrak{st}_1+X_{\text{Cut}_{n-1,1}},
\end{equation}
and by applying the DNC relation to an edge $e$ not in the $(n-1)$-cycle of $G$, we have
\begin{equation}\label{eq:cutrelation4}
X_G=X_{\text{Cut}_{n-1,1}}-X_{C_{n-1}}\mathfrak{st}_1+X_{\text{Cut}_{n-1,1}}.
\end{equation}
Now the result follows by subtracting \eqref{eq:cutrelation4} from \eqref{eq:cutrelation3}.
\item By applying the DNC relation to an edge of the cycle in $\text{Cut}_{4,n-4}$ not incident to the vertex of degree $(n-2)$, we have
\begin{equation}\label{eq:cutrelation5}
X_{\text{Cut}_{4,n-4}}=X_{\text{Cat}_{(n-2)2}}-X_{\text{Cut}_{3,n-4}}\mathfrak{st}_1+X_{\text{Cut}_{3,n-3}},
\end{equation}
and by applying the DNC relation to the edge $e$ of $G$, we have
\begin{equation}\label{eq:cutrelation6}
X_G=X_{\text{Cut}_{3,n-3}}-X_{\text{Cut}_{3,n-4}}\mathfrak{st}_1+X_{(G/e)^\multimap}.
\end{equation}
Now the result follows by subtracting \eqref{eq:cutrelation6} from \eqref{eq:cutrelation5}.
\end{enumerate}
\end{proof}

\begin{theorem}\label{thm:starcoloop}
Let $n\geq 4$ and let $G$ be a connected $n$-vertex graph other than $\text{St}_n$. Then $X_G$ can be written as a linear combination of the chromatic symmetric functions of other connected $n$-vertex graphs.
\end{theorem}

\begin{proof}
By Corollary \ref{cor:catcutbasis}, the caterpillars in \eqref{eq:catbasis} and the cuttlefish in \eqref{eq:cutbasis} form a basis for $\mathcal C_n$, so we need only consider the case where $G$ is one of these caterpillars or cuttlefish. By Part (1) of Lemma \ref{lem:cutrelations}, the chromatic symmetric function of the cuttlefish $\text{Cut}_{c,n-c}$ for $3\leq c\leq n-1$ can be expressed as a linear combination of the chromatic symmetric functions of other connected graphs, and by Part (2) of Lemma \ref{lem:cutrelations}, so can the cuttlefish $\text{Cut}_{n,0}=C_n$. By Theorem \ref{thm:treecoloops}, the chromatic symmetric function of every tree with $n\geq 8$ vertices other than $\text{St}_n$, $P_n$, and $\text{Cat}_{(n-2)2}$ can be written as a linear combination of the chromatic symmetric functions of other trees. By Parts (2) and (3) of Lemma \ref{lem:cutrelations}, the chromatic symmetric function of the path $P_n$ and caterpillar $\text{Cat}_{(n-2)2}$ for $n\geq 8$ can be written as a linear combination of the chromatic symmetric functions of other connected graphs. Finally, it is routine to check by computer that the chromatic symmetric function of the remaining caterpillars with $4\leq n\leq 7$ vertices can be written as a linear combination of the chromatic symmetric functions of other connected graphs.
\end{proof}

We conclude this section by proving Lemma \ref{lem:ab1k}. 


\begin{proof}[Proof of Lemma \ref{lem:ab1k}. ]
We begin with Part (1). We will find a linear combination $L$ of chromatic symmetric functions of trees, not using $\text{Cat}_{a1^kb}$, with the same leading partition $\lambda=ab1^k$. Then $L-a_\lambda(L)X_{\text{Cat}_{a1^kb}}$ will no longer have a $\mathfrak{st}_\lambda$ term, so its leading partition is lexicographically larger than $\lambda$, and by Corollary \ref{cor:catbasis} will lie in \begin{equation}\text{span}_{\mathbb Q}\{X_{\text{Cat}_{\mu_2\mu_3\cdots\mu_\ell\mu_1}}: \ \mu\text{ is lexicographically larger than }\lambda\}.\end{equation}
We calculate using the DNC relation and identify leading partitions using Corollary \ref{cor:catleading}. We will have several cases, depending on whether $b=2$ or $k=0$.

\textbf{Case 1: $\lambda=ab1^k$ with $a\geq b\geq 3$ and $k\geq 1$. } We have
\begin{align*}
X_{\text{Cat}_{(b-1)a1^{k-1}2}}-&X_{\text{Cat}_{a1^{k-1}2(b-1)}}\\=&(X_{\text{Cat}_{a1^{k-1}2}}\mathfrak{st}_{b-1}-X_{\text{Cat}_{(a+b-2)1^{k-1}2}}\mathfrak{st}_1+X_{\text{Cat}_{(a+b-1)1^{k-1}2}})\\&-(X_{\text{Cat}_{a1^{k-1}2}}\mathfrak{st}_{b-1}-X_{\text{Cat}_{a1^{k-1}b}}\mathfrak{st}_1+X_{\text{Cat}_{a1^{k-1}(b+1)}})\\=&X_{\text{Cat}_{a1^{k-1}b}}\mathfrak{st}_1-X_{\text{Cat}_{a1^{k-1}(b+1)}}-X_{\text{Cat}_{(a+b-2)1^{k-1}2}}\mathfrak{st}_1+X_{\text{Cat}_{(a+b-1)1^{k-1}2}},
\end{align*}
which has leading partition $\lambda=ab1^k$ by Corollary \ref{cor:catleading}.

\textbf{Case 2: $\lambda=a21^k$ with $a\geq 4$ and $k\geq 1$. } We show by induction on $k$ that the linear combination
\begin{equation*}
X_{\text{Cat}_{(a-1)21^{k-1}2}}-X_{\text{Cat}_{2(a-1)1^{k-1}2}}-X_{\text{Cat}_{(a-1)1^k3}}\end{equation*}
has leading partition $\lambda=a21^k$. If $k=1$, then 
\begin{align*}
X_{\text{Cat}_{(a-1)22}}-&X_{\text{Cat}_{2(a-1)2}}-X_{\text{Cat}_{(a-1)13}}=(X_{\text{Cat}_{(a-1)2}}\mathfrak{st}_2-X_{\text{Cat}_{(a-1)3}}\mathfrak{st}_1+X_{\text{Cat}_{(a-1)4}})\\&-(X_{\text{Cat}_{(a-1)2}}\mathfrak{st}_2-X_{\text{Cat}_{a2}}\mathfrak{st}_1+X_{\text{Cat}_{(a+1)2}})-(\mathfrak{st}_{a3}-X_{\text{Cat}_{(a-1)3}}\mathfrak{st}_1+X_{\text{Cat}_{(a-1)4}})\\&=X_{\text{Cat}_{a2}}\mathfrak{st}_1-\mathfrak{st}_{a3}-X_{\text{Cat}_{(a+1)2}},
\end{align*}
which has leading partition $\lambda=a21$ by Corollary \ref{cor:catleading}. Now if $k\geq 2$, then
\begin{align*}
X_{\text{Cat}_{(a-1)21^{k-1}2}}-&X_{\text{Cat}_{2(a-1)1^{k-1}2}}-X_{\text{Cat}_{(a-1)1^k3}}\\=&(X_{\text{Cat}_{(a-1)2}}X_{P_{k+1}}-X_{\text{Cat}_{(a-1)21^{k-2}2}}\mathfrak{st}_1+X_{\text{Cat}_{(a-1)31^{k-2}2}})\\
&-(X_{\text{Cat}_{2(a-1)}}X_{P_{k+1}}-X_{\text{Cat}_{2(a-1)1^{k-2}2}}\mathfrak{st}_1+X_{\text{Cat}_{2a1^{k-2}2}})\\&-(\mathfrak{st}_{a-1}X_{\text{Cat}_{21^{k-2}3}}-X_{\text{Cat}_{(a-1)1^{k-1}3}}\mathfrak{st}_1+X_{\text{Cat}_{a1^{k-1}3}})\\=&-\mathfrak{st}_1(X_{\text{Cat}_{(a-1)21^{k-2}2}}-X_{\text{Cat}_{2(a-1)1^{k-2}2}}-X_{\text{Cat}_{(a-1)1^{k-1}3}})\\&+(\mathfrak{st}_{a-1}X_{\text{Cat}_{31^{k-2}2}}-X_{\text{Cat}_{(a+1)1^{k-2}2}}-X_{\text{Cat}_{(a+2)1^{k-2}2}})\\&-\mathfrak{st}_{a-1}X_{\text{Cat}_{21^{k-2}3}}-X_{\text{Cat}_{2a1^{k-2}2}}-X_{\text{Cat}_{a1^{k-1}3}}\\=&-\mathfrak{st}_1(X_{\text{Cat}_{(a-1)21^{k-2}2}}-X_{\text{Cat}_{2(a-1)1^{k-2}2}}-X_{\text{Cat}_{(a-1)1^{k-1}3}})\\&-X_{\text{Cat}_{(a+1)1^{k-2}2}}-X_{\text{Cat}_{(a+2)1^{k-2}2}}-X_{\text{Cat}_{2a1^{k-2}2}}-X_{\text{Cat}_{a1^{k-1}3}},
\end{align*}
which has leading partition $\lambda=a21^k$ by our induction hypothesis and Corollary \ref{cor:catleading}. 

\textbf{Case 3: $\lambda=321^k$ with $k\geq 3$. } Let $T$ be the tree obtained by attaching two paths of length $2$ to the end of $P_{n-4}$. Let $U$ be the tree $\text{Cat}_{23}$ if $k=3$ and $\text{Cat}_{221^{k-4}2}$ if $k\geq 4$. We have
\begin{align*}
X_T&-2X_{\text{Cat}_{221^{k-1}2}}-X_{\text{Cat}_{221^{k-3}22}}=(\mathfrak{st}_2X_{\text{Cat}_{21^{k-1}2}}-X_{\text{Cat}_{221^{k-2}2}}\mathfrak{st}_1+X_{\text{Cat}_{231^{k-2}2}})\\&-(\mathfrak{st}_2X_{\text{Cat}_{21^{k-1}2}}-X_{\text{Cat}_{31^{k-1}2}}\mathfrak{st}_1+X_{\text{Cat}_{41^{k-1}2}})\\&-(X_U\mathfrak{st}_3-X_{\text{Cat}_{221^{k-2}2}}\mathfrak{st}_1+X_{\text{Cat}_{221^{k-3}22}})+X_{\text{Cat}_{221^{k-3}22}}\\=&X_{\text{Cat}_{31^{k-1}2}}\mathfrak{st}_1+X_{\text{Cat}_{231^{k-2}2}}-X_U\mathfrak{st}_3-X_{\text{Cat}_{41^{k-1}2}},
\end{align*}
which has leading partition $\lambda=321^k$ by Corollary \ref{cor:catleading}.

\textbf{Case 4: $\lambda=ab$ with $a\geq b\geq 4$. }We have
\begin{align*}
&X_{\text{Cat}_{(a-1)21(b-2)}}-X_{\text{Cat}_{(a-1)12(b-2)}}-X_{\text{Cat}_{(a-1)2(b-1)}}+X_{\text{Cat}_{(a-1)1b}}+X_{\text{Cat}_{a(b-2)2}}\\=&(X_{\text{Cat}_{a-1,2}}\mathfrak{st}_{b-1}-X_{\text{Cat}_{(a-1)2(b-2)}}\mathfrak{st}_1+X_{\text{Cat}_{(a-1)3(b-2)}})\\&-(\mathfrak{st}_aX_{\text{Cat}_{2(b-2)}}-X_{\text{Cat}_{(a-1)2(b-2)}}\mathfrak{st}_1-X_{\text{Cat}_{(a-1)3(b-2)}})\\&-(X_{\text{Cat}_{(a-1)2}}\mathfrak{st}_{b-1}-X_{\text{Cat}_{(a-1)b}}\mathfrak{st}_1-X_{\text{Cat}_{(a-1)(b+1)}})\\&+(\mathfrak{st}_{ab}-X_{\text{Cat}_{(a-1)b}}\mathfrak{st}_1+X_{\text{Cat}_{(a-1)(b+1)}})+(\mathfrak{st}_aX_{\text{Cat}_{(b-2)2}}-X_{\text{Cat}_{(a+b-3)2}}\mathfrak{st}_1+X_{\text{Cat}_{(a+b-2)2}})\\=&\mathfrak{st}_{ab}-X_{\text{Cat}_{(a+b-3)2}}\mathfrak{st}_1+X_{\text{Cat}_{(a+b-2)2}},
\end{align*}
which has leading partition $\lambda=ab$ by Corollary \ref{cor:catleading}.

\textbf{Case 5: $\lambda=a3$ with $a\geq 5$. } We have
\begin{align*}
&X_{\text{Cat}_{(a-2)212}}-X_{\text{Cat}_{(a-2)122}}+X_{\text{Cat}_{(a-2)23}}-X_{\text{Cat}_{(a-2)32}}-X_{\text{Cat}_{3(a-2)2}}+X_{\text{Cat}_{2(a-1)2}}+X_{\text{Cat}_{(a-1)4}}\\=&(X_{\text{Cat}_{(a-2)2}}\mathfrak{st}_3-X_{\text{Cat}_{(a-2)22}}\mathfrak{st}_1+X_{\text{Cat}_{(a-2)32}})-(\mathfrak{st}_{a-1}X_{\text{Cat}_{22}}-X_{\text{Cat}_{(a-2)22}}\mathfrak{st}_1+X_{\text{Cat}_{(a-2)32}})\\&+(X_{\text{Cat}_{(a-2)2}}\mathfrak{st}_3-X_{\text{Cat}_{(a-2)4}}\mathfrak{st}_1+X_{\text{Cat}_{(a-2)5}})-(X_{\text{Cat}_{(a-2)3}}\mathfrak{st}_2-X_{\text{Cat}_{(a-2)4}}\mathfrak{st}_1-X_{\text{Cat}_{(a-2)5}})\\&-(X_{\text{Cat}_{(a-2)2}}\mathfrak{st}_3-X_{\text{Cat}_{a2}}\mathfrak{st}_1-X_{\text{Cat}_{(a+1)2}})+(X_{\text{Cat}_{(a-1)2}}\mathfrak{st}_2-X_{\text{Cat}_{a2}}\mathfrak{st}_1+X_{\text{Cat}_{(a+1)2}})\\&+X_{\text{Cat}_{(a-1)4}}\\=&X_{\text{Cat}_{(a-2)2}}\mathfrak{st}_3-\mathfrak{st}_{a-1}X_{\text{Cat}_{22}}-X_{\text{Cat}_{(a-2)3}}\mathfrak{st}_2+X_{\text{Cat}_{(a-1)2}}\mathfrak{st}_2+X_{\text{Cat}_{(a-1)4}}\\=&(\mathfrak{st}_{(a-2)32}-\mathfrak{st}_{(a-1)31}+\mathfrak{st}_{a3})-(\mathfrak{st}_{(a-1)22}-\mathfrak{st}_{(a-1)31}+\mathfrak{st}_{(a-1)4})-(\mathfrak{st}_{(a-2)32}-\mathfrak{st}_{a21}+\mathfrak{st}_{(a+1)2})\\&+(\mathfrak{st}_{(a-1)22}-\mathfrak{st}_{a21}+\mathfrak{st}_{(a+1)2})+(\mathfrak{st}_{(a-1)4}-\mathfrak{st}_{(a+2)1}+\mathfrak{st}_{a3})\\=&\mathfrak{st}_{a3}-\mathfrak{st}_{(a+2)1}+\mathfrak{st}_{(a+3)}=X_{\text{Cat}_{a3}}.
\end{align*}
Finally, we prove Part (2). Once again, it suffices to find a linear combination of the chromatic symmetric function of trees, not using $\text{Cat}_{aaa}$, with leading partition $\lambda=aaa$. Let $T$ be the tree obtained by attaching a path of length $2$ to the middle vertex of $\text{Cat}_{a(a-2)a}$. Then we have
\begin{align*}
X_T-&X_{\text{Cat}_{a(a-2)a2}}-X_{\text{Cat}_{a1a(a-1)}}=(X_{\text{Cat}_{a(a-2)a}}\mathfrak{st}_2-X_{\text{Cat}_{a(a-1)a}}\mathfrak{st}_1+X_{\text{Cat}_{aaa}})\\&-(X_{\text{Cat}_{a(a-2)a}}\mathfrak{st}_2-X_{\text{Cat}_{a(a-2)(a+1)}}\mathfrak{st}_1+X_{\text{Cat}_{a(a-2)(a+2)}})\\&-(\mathfrak{st}_aX_{\text{Cat}_{(a+1)(a-1)}}-X_{\text{Cat}_{aa(a-1)}}\mathfrak{st}_1+X_{\text{Cat}_{(a+1)a(a-1)}})\\=&-\mathfrak{st}_1(X_{\text{Cat}_{a(a-1)}}\mathfrak{st}_a-X_{\text{Cat}_{a(2a-2)}}\mathfrak{st}_1+X_{\text{Cat}_{a(2a-1)}})+X_{\text{Cat}_{aaa}}\\&+X_{\text{Cat}_{a(a-2)(a+1)}}\mathfrak{st}_1-X_{\text{Cat}_{a(a-2)(a+2)}}-\mathfrak{st}_aX_{\text{Cat}_{(a+1)(a-1)}}\\&+\mathfrak{st}_1(\mathfrak{st}_aX_{\text{Cat}_{a(a-1)}}-X_{\text{Cat}_{(2a-1)(a-1)}}\mathfrak{st}_1+X_{\text{Cat}_{(2a)(a-1)}})-X_{\text{Cat}_{(a+1)a(a-1)}}\\=&X_{\text{Cat}_{aaa}}-X_{\text{Cat}_{a(a-2)(a+2)}}-\mathfrak{st}_aX_{\text{Cat}_{(a+1)(a-1)}}-X_{\text{Cat}_{(a+1)a(a-1)}}\\&+\mathfrak{st}_1(X_{\text{Cat}_{a(2a-2)}}\mathfrak{st}_1-X_{\text{Cat}_{a(2a-1)}}+X_{\text{Cat}_{a(a-2)(a+1)}}-X_{\text{Cat}_{(2a-1)(a-1)}}\mathfrak{st}_1+X_{\text{Cat}_{(2a)(a-1)}}),
\end{align*}
which has leading partition $\lambda=aaa$ by Corollary \ref{cor:catleading}.
\end{proof}

\section{The coefficient of $\mathfrak{st}_n$ and acyclic orientations}\label{sec:coefn}

In this section, we give an interpretation of the coefficient $c_n$ in $X_G=\sum_\lambda c_\lambda\mathfrak{st}_\lambda$ in terms of \emph{acyclic orientations} of $G$. 

\begin{definition}
An \emph{acyclic orientation} $\theta$ of a graph $G=(V,E)$ is an assignment of directions to the edges so that there are no directed cycles. A vertex $v\in V$ is a \emph{sink} of $\theta$ if all edges incident to $v$ are directed toward $v$.
\end{definition}

Stanley \cite[Corollary 1.3]{acyclic} proved that the number of acyclic orientations of $G$ is $|\chi_G(-1)|$. Greene and Zaslavsky proved that the linear coefficient of the chromatic polynomial counts acyclic orientations of $G$ with a particular sink.

\begin{theorem}\label{thm:chiacyclic} \cite[Theorem 7.3]{interpwhitney} Let $G=(V,E)$ be a graph and fix a vertex $v\in V$. Then $|\chi_G'(0)|$ is the number of acyclic orientations of $G$ with unique sink $v$. In particular, this number does not depend on $v$.
\end{theorem}

Stanley further showed that the $e$-expansion of the chromatic symmetric function encodes even more information about acyclic orientations.

\begin{theorem}\label{thm:xacyclic} \cite[Theorem 3.3]{chromsym}
Let $G=(V,E)$ be a graph, let $X_G=\sum_\lambda d_\lambda e_\lambda$, and fix an integer $k$. Then 
\begin{equation}
\sum_{\ell(\lambda)=k}d_\lambda=\text{the number of acyclic orientations of }G\text{ with exactly }k\text{ sinks.}
\end{equation}
\end{theorem}

In particular, the coefficient of $e_n$ is equal to the number of acyclic orientations with a unique sink, which by Theorem \ref{thm:chiacyclic} is equal to $n|\chi_G'(0)|$. We now show a similar result about the coefficient of $\mathfrak{st}_n$.

\begin{theorem}
Let $G=(V,E)$ be an $n$-vertex graph, let $X_G=\sum_\lambda c_\lambda\mathfrak{st}_\lambda$, and fix a vertex $v\in V$. Then $c_n$ is the number of acyclic orientations of $G$ with unique sink $v$. 
\end{theorem}

\begin{proof}
By Theorem \ref{thm:xacyclic}, the coefficient of $e_n$ in $\mathfrak{st}_\lambda$ is the number of acyclic orientations of the star forest $\text{St}_\lambda$ with a unique sink. Because every connected component must have a sink, this coefficient is $0$ unless $\lambda=n$, in which case it is $n$. Therefore, if we write $X_G=\sum_\lambda c_\lambda\mathfrak{st}_\lambda=\sum_\mu d_\mu e_\mu$, we see that $c_n=\frac 1n d_n=|\chi_G'(0)|$ is the number of acyclic orientations of $G$ with unique sink $v$. 
\end{proof}

This allows us to recover some known results connecting the coefficient $c_n$ to properties of the graph $G$.

\begin{corollary} \label{cor:coefn} Let $G=(V,E)$ be an $n$-vertex graph and let $X_G=\sum_\lambda c_\lambda\mathfrak{st}_\lambda$. 
\begin{enumerate}
\item We have $c_n\neq 0$ if and only if $G$ is connected.
\item \cite[Section 3]{chromstar} We have $c_n=1$ if and only if $G$ is a tree.
\item \cite[Corollary 3.4]{chromunicyclic} If $G$ is a connected unicyclic graph with a cycle of size $c$, then $c_n=c-1$. 
\end{enumerate}
\end{corollary}

\begin{proof} In each case, we fix a vertex $v\in V$ and count acyclic orientations with unique sink $v$. 
\begin{enumerate}
\item If $G$ is not connected, then every acyclic orientation of $G$ will have at least one sink in each component of $G$, so $c_n=0$. Conversely, if $G$ is connected, then fix a vertex $v\in V$ and order the vertices of $G$ as follows. We start with $v$, and at each step, select a vertex adjacent to a previously chosen one. Because $G$ is connected, every vertex will be selected at some point. Then we define an acyclic orientation of $G$ by orienting each edge toward the vertex chosen earlier. This acyclic orientation will have unique sink $v$, so $c_n\geq 1$.
\item If $G$ is a tree and $v\in V$ is fixed, then the only acyclic orientation with unique sink $v$ is that where each edge is directed toward the vertex closer to $v$. Conversely, if $G$ is a connected graph with a cycle $C$, fix a vertex $v$ of the cycle and let $u$ and $w$ be its neighbours. Because $v$ is the unique sink, the edges $\{u,v\}$ and $\{v,w\}$ must be directed toward $v$. The other edges of $C$ can be directed as either a path from $u$ to $w$ or a path from $w$ to $u$. Therefore, there are at least two acyclic orientations of $C$ with unique sink $v$. These acyclic orientations of $C$ can be extended to those of $G$ by ordering the remaining vertices as in Part (1).
\item Fix a vertex $v$ of the cycle $C$ of $G$. There are exactly $(c-1)$ acyclic orientations of $C$ with unique sink $v$, given by choosing one of the other vertices $w$ of $C$ and directing the edges as paths from $w$ to $v$. These $(c-1)$ acyclic orientations each have a unique extension to $G$ by directing each edge toward the vertex closer to the cycle.
\end{enumerate}
\end{proof}

Cho and van Willigenburg described a general method for constructing bases of $\Lambda_n$ using chromatic symmetric functions.

\begin{theorem} \cite[Theorem 5]{chrombases}
Fix a family $\mathcal F=(G_1,G_2,\ldots)$ of connected graphs where $G_i$ has $i$ vertices for every $i$. For a partition $\lambda\vdash n$, let $G_\lambda$ be the disjoint union of $G_{\lambda_1},\ldots,G_{\lambda_\ell}$. Then the set $B_{\mathcal F}=\{X_{G_\lambda}: \ \lambda\vdash n\}$ is a basis of $\Lambda_n$. 
\end{theorem}

In particular, we can write every chromatic symmetric function $X_G$ as a rational linear combination of the $X_{G_\lambda}$. We could ask which $\mathcal F$ have the property that we can write every chromatic symmetric function $X_G$ as an \emph{integer} linear combination of the $X_{G_\lambda}$. Our interpretation of $c_n$ will give us a simple answer.

\begin{theorem}
Fix a family $\mathcal F=(G_1,G_2,\ldots)$ of connected graphs where $G_i$ has $i$ vertices for every $i$. Every chromatic symmetric function $X_G$ can be written as an integer linear combination of the $X_{G_\lambda}$ if and only if every $G_i$ is a tree. 
\end{theorem}

\begin{proof} We first suppose that $\mathcal F$ is a family of trees. Let $G$ be an $n$-vertex graph. Recall that we can write $X_G$ as an integer combination in the star basis because the coefficient $c_\lambda$ of $\mathfrak{st}_\lambda$ is equal to $(-1)^{a_1}$ multiplied by the number of paths in a DNC tree from $G$ to $\text{St}_\lambda$. We now consider the change-of-basis between the star basis and the basis $B_{\mathcal F}$. By Corollary \ref{cor:coefn}, Part (2), the star expansion of $X_{G_i}$ is of the form
\begin{equation}
X_{G_i}=\mathfrak{st}_i+\sum_{\lambda\vdash i, \ \lambda\neq i}c_\lambda(G_i)\mathfrak{st}_\lambda
\end{equation}
for some coefficients $c_\lambda(G_i)\in\mathbb Z$. Therefore, for a partition $\mu\vdash n$, the star expansion of $X_\mu=X_{G_{\mu_1}}\cdots X_{G_{\mu_\ell}}$ is of the form
\begin{equation} \label{eq:order}
X_\mu=\mathfrak{st}_\mu+\sum_{\lambda\vdash n, \ \lambda<_{\text{lex}}\mu}c_\lambda\mathfrak{st}_\lambda.
\end{equation}
Let $M$ be the change-of-basis matrix from the basis $B_{\mathcal F}$ to the star basis, where we order the indexing partitions so that they decrease lexicographically. By Equation \eqref{eq:order}, $M$ has integer entries and is upper triangular with $1$'s on the main diagonal. Therefore, the inverse matrix $M^{-1}$ also has integer entries, so we can write any $\mathfrak{st}_\mu$ as an integer combination of the $X_\lambda$. Because $X_G$ can be written as an integer combination of the $\mathfrak{st}_\mu$, it can be written as an integer combination of the $X_\lambda$.

Now suppose that $\mathcal F$ is not a family of trees, and let $n$ be minimal such that $G_n$ has a cycle. By Corollary \ref{cor:coefn}, Part (2), we have \begin{equation}X_{G_n}=c_n(G_n)\mathfrak{st}_n+\sum_{\lambda\vdash n, \ \lambda\neq n}c_\lambda(G_n)\mathfrak{st}_\lambda\end{equation}
for some coefficients $c_\lambda(G_n)\in\mathbb Z$ and $c_n(G_n)\geq 2$. By Corollary \ref{cor:coefn}, Part (1), the star $\mathfrak{st}_n$ does not appear in the star expansion of any other $X_{G_\mu}$ for $\mu\neq n$. Therefore, in the expansion of $\mathfrak{st}_n$ in the basis $B_{\mathcal F}$, the coefficient of $G_n$ is $\frac 1{c_n(G_n)}\notin\mathbb Z$, so we cannot write $\mathfrak{st}_n$ as an integer combination of the $X_{G_\lambda}$. 
\end{proof}

This result suggests that if we seek combinatorial interpretations of the coefficients of chromatic symmetric functions $X_G$ in these chromatic bases $B_{\mathcal F}$, it may be best to take families of trees.

\section{Statements and Declarations}
R. Orellana was partially supported by NSF grant DMS--2452044. This material is also based upon work
supported by the National Science Foundation, while R. Orellana was in residence
at the ICERM semester program ``Categorification and Computation in Algebraic Combinatorics" in Fall
2025. The authors have no relevant financial or non-financial interests to disclose.



\end{document}